\documentclass[11pt,a4paper]{article}

\usepackage[utf8]{inputenc}
\usepackage[english]{babel}

\usepackage{geometry}
\usepackage{graphicx}
\usepackage{amssymb,amsmath,amsthm}
\usepackage{marvosym}
\usepackage{subcaption}
\usepackage{delarray}
\usepackage{hyperref}
\usepackage{tikz}
\graphicspath{{figures/}}
\usepackage{authblk}

\usepackage{mathtools}

\newcommand{\N}{\mathbb{N}}
\newcommand{\Z}{\mathbb{Z}}
\newcommand{\R}{\mathbb{R}}

\newcommand{\0}{\mathbb{0}}

\newcommand{\Poly}{{\mathsf{P}}}
\newcommand{\NC}{{\mathsf{NC}}}

\newcommand{\neighborhood}{\mathcal{N}}
\newcommand{\degree}{d}
\newcommand{\neighbors}{\mathcal{N}}
\newcommand{\indic}{{\bf 1}}
\newcommand{\unstab}{Act}
\newcommand{\stab}{Stab}

\newtheorem{theorem}{\textbf{Theorem}}
\newtheorem{lemma}{\textbf{Lemma}}
\newtheorem{proposition}{\textbf{Proposition}}

\newtheorem{corollary}{\textbf{Corollary}}
\newtheorem{remark}{\textbf{Remark}}

\newtheorem{definition}{\textbf{Definition}}

\newtheorem{remark*}{\textbf{Remark}}

\title{Any shape can ultimately cross information\\ on two-dimensional abelian sandpile models}
\author[1,2]{Viet-Ha Nguyen}
\author[2]{Kevin Perrot}
\affil[1]{\'Ecole Normale Sup\'erieure de Lyon, Computer Science department, Lyon, France}
\affil[2]{Aix-Marseille Universit\'e, CNRS, Centrale Marseille, LIF, Marseille, France}
\date{}

\begin{document}
\renewcommand{\labelitemi}{$\bullet$}
\renewcommand{\labelitemii}{$\circ$}
\maketitle

%%%%%%%%%%%%%%%%%%%%%%%%%%
\begin{abstract}
  In this paper we study the abelian sandpile model on the two-dimensional grid with uniform neighborhood, and prove that any family of neighborhoods defined as scalings of a continuous non-flat shape can ultimately perform crossing.
\end{abstract}

%%%%%%%%%%%%%%%%%%%%%%%%%%
\section{Introduction}

In \cite{1987-BakTangWiesenfeld-SOC}, three physicists proposed the now famous two-dimensional {\em abelian sandpile model} with von Neumann neighborhood of radius one. This number-conserving discrete dynamical system is defined by a simple local rule describing the movements of sand grains in the discrete plane $\Z^2$, and exhibits surprisingly complex global behaviors.

The model has been generalized to any directed graph (\cite{1992-Lovasz,1998-Dhar-AbelianSanpileAndRelatedModels}). Basically, given a digraph, each vertex has a number of sand grains on it, and a vertex that has more grains than out-neighbors can give one grain to each of its out-neighbors. This model is Turing-universal (\cite{1997-Goles-UniversalityCFG}). When restricted to particular directed graphs (digraphs), an interesting notion of complexity is given by the following {\em prediction problem}.

\vspace*{.5em}
\fbox{\parbox{12.6cm}{
  {\bf Prediction problem.}\\
  {\it Input:} a finite and stable configuration, and two vertices $v$ and $u$.\\
  {\it Question:} \hspace*{-.2cm}\begin{tabular}[t]{l}does adding one grain on vertex $v$ triggers a chain of reactions\\that will reach vertex $u$?\end{tabular}
}}
\vspace*{.5em}

Depending on the restrictions applied to the digraph, the computational complexity in time of this problem has sometimes been proven to be $\Poly$-complete, and sometimes to be in $\NC$. In order to prove the $\Poly$-completeness of the prediction problem, authors naturally try to implement circuit computations, via reductions from the {\em Monotone Circuit Value Problem} (MCVP), {\em i.e.} they show how to implement the following set of gates: {\em wire}, {\em turn}, {\em multiply}, {\em and}, {\em or}, and {\em crossing} (see Section \ref{s:known} for a review).

In abelian sandpile models, monotone gates are usually easy to implement with {\em wires} constructed from sequences of vertices that fire one after the other\footnote{this is a particular case of {\em signal} ({\em i.e.} information transport) that we can qualify as {\em elementary}.}: an {\em or} gate is a vertex that needs one of its in-neighbors to fire; an {\em and} gate is a vertex that needs two of its in-neighbors to fire. The crucial part in the reduction is therefore the implementation of a {\em crossing} between two wires. Regarding regular graphs, the most relevant case is the two-dimensional grid (in dimension one crossing is less meaningful, and from dimension three is it easy to perform a crossing using an extra dimension; see Section \ref{s:known} for references).

When it is possible to implement a crossing, then the prediction problem is $\Poly$-complete. The question is now to formally relate the impossibility to perform a crossing with the computational complexity of the prediction problem (would it be in $\NC$?). The goal is thus to find conditions on a neighborhood so that it cannot perform a crossing (this requires a precise definition of {\em crossing}), and prove that these conditions also imply that the prediction problem is in $\NC$. As an hint for the existence of such a link, it is proven in \cite{2006-Goles-CrossingInfo2DBTWSandpile} that crossing information is not possible with von Neumann neighborhood of radius one, for which the computational complexity of the prediction problem has never been proven to be $\Poly$-complete (neither in $\NC$). The present work continues the study on general uniform neighborhoods, and shows that the conditions on the neighborhood so that it can or cannot perform crossing are intrinsically discrete. 

Section \ref{s:def} defines the abelian sandpile model, neighborhood, shape, and crossing configuration (this last one requires a substantial number of elements to be defined with precision, as it is one of our aims), and Section \ref{s:known} reviews the main known results related to prediction problem and information crossing. The notion of firing graph (from \cite{2006-Goles-CrossingInfo2DBTWSandpile}) is presented and studied at the beginning of Section \ref{s:study}, which then establishes some conditions on crossing configurations for convex neighborhoods, and finally exposes the main result of this paper: that any shape can ultimately perform crossing.

%%%%%%%%%%%%%%%%%%%%%%%%%%
\section{Definitions}
\label{s:def}

In the literature, {\em abelian sandpile model} and {\em chip-firing game} usually refer to the same discrete dynamical system, sometimes on different classes of (un)directed graphs.

\subsection{Abelian sandpile models on $\Z^2$ with uniform neighborhood}

Given a digraph $G=(V,A)$, we denote $\degree^+(v)$ (resp. $\degree^-(v)$) the out-degree (resp. in-degree) of vertex $v \in V$, and $\neighbors^+(v)$ (resp. $\neighbors^-(v)$) its set of out-neighbors (resp. in-neighbors). A {\em configuration} $c$ is an assignment of a finite number of sand grains to each vertex, $c : V \to \N$. The dynamics $F:\N^{|V|} \to \N^{|V|}$ is defined by the parallel application of a local rule at each vertex: if vertex $v$ contains at least $\degree^+(v)$ grains, then it gives one grain to each of its out-neighbors (we say that $v$ {\em fires}, or $v$ is a {\em firing} vertex). Formally,
\begin{equation}
  \label{eq:sandpile}
  \forall v \in V,~ \big(F(c)\big)(v)=c(v)-\degree^+(v)\indic_{\N}\big(c(v)-\degree^+(v)\big)+\sum_{u \in \neighbors^-(v)} \indic_{\N}\big(c(u)-\degree^+(u)\big)
\end{equation}
in which $\indic_{\N}(x)$ the indicator function of $\N$, that equals 1 when $x \geq 0$ and 0 when $x<0$. Note that this discrete dynamical system is deterministic. An example of evolution is given on Figure \ref{fig:sandpile}.

\begin{figure}[!h]
  \centering \includegraphics{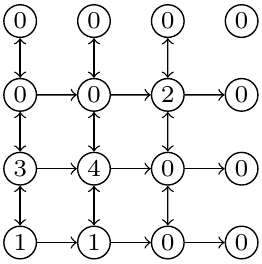} \raisebox{1.2cm}{$\overset{F}{\mapsto}$} \includegraphics{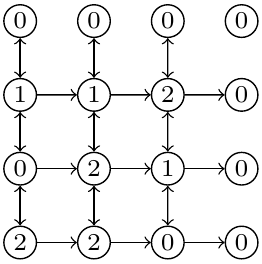} \raisebox{1.2cm}{$\overset{F}{\mapsto}$} \includegraphics{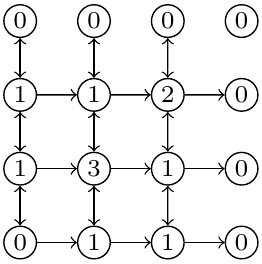}
  \caption{An example of two evolution steps in the abelian sandpile model.}
  \label{fig:sandpile}
\end{figure}

\begin{remark}
  \label{remark:self-loops}
  As self-loops (vertices of the form $(v,v)$ for some $v \in V$) are not useful for the dynamics (there are just some grains trapped on the vertex), all the digraphs we consider won't have self-loops, even when we don't explicitly mention this.
\end{remark}

We say that a vertex $v$ is {\em stable} when $c(v)<\degree^+(v)$, and {\em unstable} otherwise. By extension, a configuration $c$ is {\em stable} when all the vertices are stable, and {\em unstable} if at least one vertex is not stable. Given a configuration $c$, we denote $\stab(c)$ (resp. $\unstab(c)$) the set of stable (resp. unstable) vertices.

In this work, we are interested in the dynamics when the support graph is the two-dimensional grid $\Z^2$, with a uniform neighborhood. In mathematical terms, given some finite {\em neighborhood} $\neighborhood^+ \subset \Z^2$, we define the graph $G^{\neighborhood^+}=(V,A^{\neighborhood^+})$ on which we will study the abelian sandpile dynamics, with $V=\Z^2$ and
\begin{equation}
  \label{eq:neighbors}
  A^{\neighborhood^+}=\big\{ \big((x,y),(x',y')\big) ~|~ (x'-x,y'-y) \in \neighborhood^+ \big\}.
\end{equation}
On $G^{\neighborhood^+}$ a vertex fires if it has at least $p^{\neighborhood^+}=|\neighborhood^+|$ grains. When there is no ambiguity, we will omit the superscript $\neighborhood^+$ in order to lighten the notations. An example is given on Figure \ref{fig:graph-n}.

\begin{figure}[!h]
  \centering \raisebox{1cm}{\includegraphics{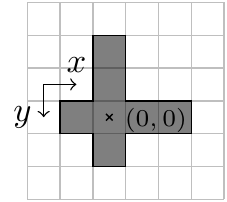}} \hspace*{1cm} \includegraphics{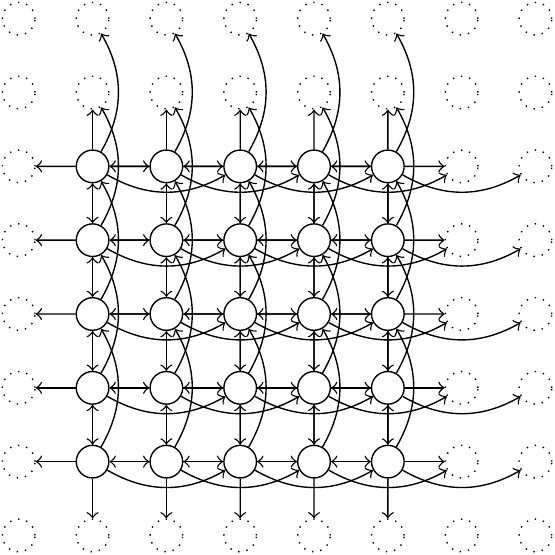}
  \caption{A neighborhood $\neighborhood^+$ (left) and a part of the corresponding graph $G^{\neighborhood^+}$ (right). In this example $p^{\neighborhood^+}=6$ grains.}
  \label{fig:graph-n}
\end{figure}

We say that a configuration is {\em finite} when it contains a finite number of grains, or equivalently when the number of non-empty vertices is finite (by definition, the number of grains on each vertex is finite). We say that a finite configuration $c$ is a square of size $n \times n$ if there is no grain outside a window of size $n$ by $n$ cells: there exists $(x_0,y_0)$ such that for all $(x,y) \in \Z^2 \setminus \{ (x',y') ~|~ x_0 \leq x' < x_0+n  \wedge y_0 \leq y' < y_0+n \}$ we have $c((x,y))=0$.

\begin{definition}[movement vector]
  Given a {\em neighborhood} $\neighborhood^+ \subset \Z^2 \setminus \{(0,0)\}$ of $p$ cells, a vector $\vec{v}$ such that $(0,0)+\vec{v} \in \neighborhood^+$ is called a {\em movement vector}. We denote $\neighborhood^+(v)$ the set of neighbors of vertex $v$, {\em i.e.} $\neighborhood^+(v)=\neighborhood^++\vec{v}$.
\end{definition}

We will only study finite neighborhoods and finite configurations, which ensures that the dynamic converges when the graph is connected (otherwise the neighborhood has only collinear movement vectors and crossing information becomes less meaningful), as stated in the following lemma.

\begin{lemma}
  \label{lemma:converge}
  Given a finite neighborhood $\neighborhood^+$ with at least two non-collinear movement vectors $\vec{u}, \vec{v}$, and a finite configuration $c$, there exists $t_0 \in \N$ such that $F^{t_0}(c)$ is stable.
\end{lemma}

\begin{proof}
  For the contradiction, suppose that $c$ never converges to a stable configuration. Then there exists an infinite sequence of vertices $(v_j)_{j \in \N}$ that fire. We consider two cases: either there exists a vertex that fires infinitely often, or there exists an infinity of different vertices that fire.

  If there exists a vertex $v$ that occurs infinitely often in $(v_j)_{j \in \N}$, then $v$ sends an infinity of grains to vertex $v+\vec{u}$, which sends an infinity of grains to $v+2\vec{u}$, etc. This contradicts the fact that there are finitely many grains in $c$ (the number of grains is constant throughout the evolution), since those grains never come back to $v$.

  If there exists an infinity of different vertices that fire, then let us consider a rectangle $R$ of finite size such that any vertex $v$ outside $R$ is such that $c(v)=0$. Any vertex $v$ outside $R$ that is fired at time step $t_1$ needs all its in-neighbors to be fired strictly before it in order to be fired. In particular, it requires vertices $v-\vec{u}$ and $v-\vec{v}$ to be fired strictly before it. By induction, since $\vec{u}$ and $\vec{v}$ are not collinear, if $v$ is far enough from $R$  (we supposed that that there are infinitely many different vertices that fire) then there exist one of $\vec{u}$ and $\vec{v}$ such that $\{ v-k\vec{u} ~|~ k \in \N\} \cap R = \emptyset$ or $\{ v-k\vec{v} ~|~ k \in \N\} \cap R = \emptyset$, which is an infinite sequence of vertices that all need to be fired strictly before the other, which contradicts the fact that $v$ is fired at some finite time step $t_1$.
\end{proof}

Lemma \ref{lemma:converge} allows to study any sequential evolution (where one unstable vertex is non-deterministically chosen to fire at each time step) of the abelian sandpile model, since when the dynamics converges to a stable configuration, any sequential evolution converges to the same stable configuration and every vertex is fired exactly the same number of times (this fact is related to the {\em abelian} property, see for example \cite{2008-LevinePeresPropp-CFGRotorRouterDirectedGraphs}).

Finaly, there is a natural notion of addition among configurations on the same set of vertices. Given $c,c'$ two configurations on some set of vertices $V$, we define the configuration $(c+c')$ as $(c+c')(v)=c(v)+c'(v)$ for all $v \in V$.

\subsection{Shape of neighborhood}

A shape will be defined as a continuous area in $\R^2$, that can be placed on the grid to get a discrete neighborhood $\neighborhood^+$ that defines a graph $G^{\neighborhood^+}$ for the abelian sandpile model.

\begin{definition}[shape]
  \label{def:shape}
  A {\em shape} (at $(0,0)$) is a bounded set $s^+ \subset \R^2$. We define the {\em neighborhood $\neighborhood^+_{s^+,r}$ of shape $s^+$} (with the firing cell at $(0,0)$) {\em with scaling ratio $r \in \R$}, $r > 0$, as 
  $$\neighborhood^+_{s^+,r} = \{(x,y) \in \Z^2 ~|~ (x/r,y/r) \in s^+\} \setminus \{(0,0)\}.$$
  We also have {\em movement vectors} $\vec{v}$ such that $(0,0)+\vec{v} \in s^+$, and denote $s^+(v) = s^+ + \vec{v}$.
\end{definition}

A {\em partition} $S^1, S^2, \dots, S^k$ of a set $S$ (either in $\Z^2$ or in $\R^2$) is such that $\bigcup_{i=1}^{k}{S^i} = S$ and for all $i \neq j$, $S_i \cap S_j = \emptyset$. Given a neighborhood (resp. a shape) $S$, $S_i$ is called a {\em subneighborhood} (resp. a {\em subshape}) of $S$.

%\begin{definition}
%  A shape $s$ is connected if there does not exist a partition of $s$ such that there are two distinct subsets.
%\end{definition}

We recall Remark \ref{remark:self-loops}: self-loops are removed from the dynamics. A shape is bounded so that its corresponding neighborhoods are finite ({\em i.e.} there is a finite number of neighbors). An example of shape is given on Figure \ref{fig:shape}.

\begin{figure}[!h]
  \centering \includegraphics{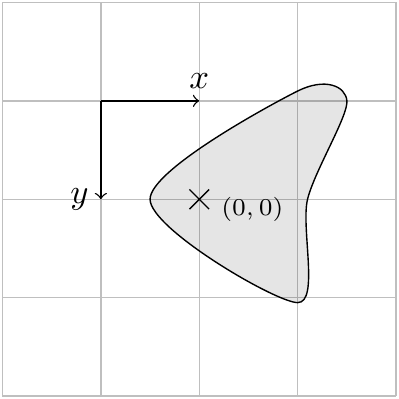} \hspace*{1cm} \includegraphics{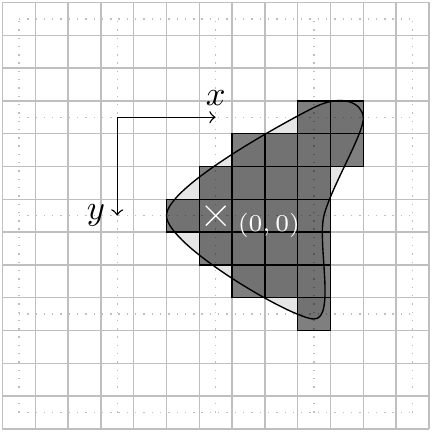}
  \caption{A shape $s^+$ on $\R^2$ (left, a discrete grid is displayed to see the 1:1 scale of the shape), and the neighborhood $\neighborhood^+_{s^+,3}$ (right, dotted lines reproduce the original grid from the left picture, and the discrete neighborhood in $\Z^2$ is darken).}
  \label{fig:shape}
\end{figure}

\begin{remark}
  A given neighborhood $\neighborhood^+ \subset Z^2$ always corresponds to an infinity of shapes and scaling ratio: for all $\neighborhood^+$, $|\{(s^+,r) ~|~ \neighborhood^+_{s^+,r}=\neighborhood^+ \}|=+\infty$.
\end{remark}

The notion of inverse shape and inverse neighborhood will be of interest in the analysis of Section \ref{s:study}: it defines the set of cells which have a given cell in their neighborhood (the neighboring relation is not symmetric).

\begin{definition}[inverse]
  The inverse $\neighborhood^{-}$ (resp. $s^{-}$) of a neighborhood $\neighborhood^+$ (resp. of a shape $s^+$) is defined via the central symmetry around $(0,0)$,
  $$\neighborhood^{-}=\left\{ (x,y) \in \Z^2 ~|~ (-x,-y) \in \neighborhood^+ \right\} \text{ and } s^{-}=\left\{ (x,y) \in \R^2 ~|~ (-x,-y) \in s^+ \right\}.$$
\end{definition}

\begin{remark}
   For any shape $s^+$ and any ratio $r>0$, we have $\neighborhood^{-}_{s^+,r}=\neighborhood^+_{s^{-1},r}$.
\end{remark}

We also have the inverse shape $s^{-}(v)$ at any point $v \in \R^2$ and the inverse neighborhood $\neighborhood^{-}(v)$ at any point $v \in \Z^2$. For any $u,v \in \Z^2$ (resp. $\R^2$),
$$v \in \neighborhood^+(u) \iff u \in \neighborhood^{-}(v) \text{ ~~~~~(resp. } v \in s^+(u) \iff u \in s^{-}(v) \text{)}.$$

We want shapes to have some thickness everywhere, as stated in the next definition. We denote $T_{(x,y),(x',y'),(x'',y'')}$ the triangle of coordinates $(x,y),(x',y'),(x'',y'') \in \R^2$.

\begin{definition}[non-flat shape]
  We say that a shape $s^+$ is {\em non-flat} when for every point $(x,y) \in s^+$ there exist $(x',y'), (x'',y'') \in \R^2$ such that the triangle $T_{(x,y),(x',y'),(x'',y'')}$ has a strictly positive area ({\em i.e.} the three points are not aligned), and entirely belongs to $s^+$.
\end{definition}

\subsection{Crossing configuration}

These definitions are inspired by \cite{2006-Goles-CrossingInfo2DBTWSandpile}.

A {\em crossing configuration} will be a finite configuration, and for convenience with the definition we take it of size $n \times n$ for some $n \in \N$, with non-empty vertices inside the square from $(0,0)$ to $(n-1,n-1)$ (see Figure \ref{fig:orientation}). The idea is to be able to add a grain on the west border to create a chain of reactions that reaches the east border, and a grain on the north border to create a chain of reactions that reaches the south border.

\begin{figure}[!h]
  \centering \includegraphics{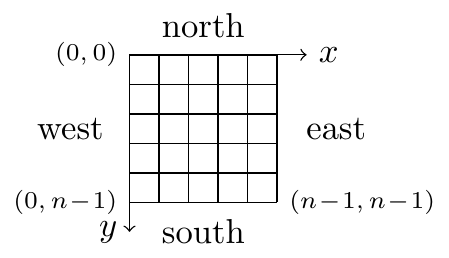}
  \caption{Orientation and positioning of an $n \times n$ square ($n=5$).}
  \label{fig:orientation}
\end{figure}

Let $E_n \subset \{0,1\}^n$ be the set of vectors that contain exactly one value 1, that is $E_n=\{ \vec{e} \in \{0,1\}^n ~|~ \exists ! i \text{ such that } \vec{e}(i)=1 \}$. In order to convert vectors to configurations, we define four positionings of a given vector $\vec{e} \in \{0,1\}^n$: $N(\vec{e})$, $W(\vec{e})$, $S(\vec{e})$ and $E(\vec{e})$ are four configurations of size $n \times n$, defined as
$$
  \begin{array}{rcl}
  N(\vec{e}) &:& (x,y) \mapsto \begin{array}\{{ll}.\vec{e}(x) & \text{if } y=0\\0 & \text{otherwise}\end{array}\\
    E(\vec{e})) &:& (x,y) \mapsto \begin{array}\{{ll}.\vec{e}(y) & \text{if } x=n-1\\0 & \text{otherwise}\end{array}\\
    S(\vec{e}) &:& (x,y) \mapsto \begin{array}\{{ll}.\vec{e}(x) & \text{if } y=n-1\\0 & \text{otherwise}\end{array}\\
    W(\vec{e}) &:& (x,y) \mapsto \begin{array}\{{ll}.\vec{e}(y) & \text{if } x=0\\0 & \text{otherwise}\end{array}
  \end{array}
$$

\begin{definition}[transporter]
  We say that a finite configuration $c$ of size $n \times n$ is a {\em transporter from west to east with vectors} $\vec{w},\vec{e} \in E_n$ when
  \begin{enumerate}
    \item $c$ is stable;
    \item $\exists t \in \N,~ \unstab(F^t(c+W(\vec{w})))=\{ v \in \Z^2 ~|~ E(\vec{e})(v)=1 \}$.
  \end{enumerate}
  We define symmetrically a configuration that is a {\em transporter from north to south with vectors} $\vec{n},\vec{s} \in E_n$ when
  \begin{enumerate}
    \item $c$ is stable;
    \item $\exists t \in \N,~ \unstab(F^t(c+N(\vec{n})))=\{ v \in \Z^2 ~|~ S(\vec{s})(v)=1 \}$.
  \end{enumerate}
\end{definition}

Besides transport of a signal (implemented via firings) from one border to the other (from west to east, and from north to south), a proper crossing of signals must not fire any cell on the other border: the transport from west to east must not fire any cell on the south border, and the transport from north to south must not fire any cell on the east border. This is the notion of isolation presented in the next definition.

\begin{definition}[isolation]
  We say that a finite configuration $c$ of size $n \times n$ has {\em west vector $\vec{w} \in E_n$ isolated to the south} when
  \begin{enumerate}
    \item $\forall t \in \N,~ Act(F^t(c+W(\vec{w}))) \cap \{ (x,y) ~|~ y=n-1 \} = \emptyset$.
  \end{enumerate}
  We define symmetrically a configuration that has {\em north vector $\vec{n} \in E_n$ isolated to the east} when
  \begin{enumerate}
    \item $\forall t \in \N,~ Act(F^t(c+N(\vec{n}))) \cap \{ (x,y) ~|~ x=n-1 \} = \emptyset$.
  \end{enumerate}
\end{definition}

\begin{definition}[crossing configuration]
  A finite configuration $c$ of size $n \times n$ is a {\em crossing with vectors $\vec{n},\vec{e},\vec{s},\vec{w} \in E_n$} when
  \begin{enumerate}
    \item $c$ is stable;
    \item $c$ is a transporter from west to east with vectors $\vec{w},\vec{e}$;
    \item $c$ has west vector $\vec{w}$ isolated to the south;
    \item $c$ is a transporter from north to south with vectors $\vec{n},\vec{s}$;
    \item $c$ has north vector $\vec{n}$ isolated to the north.
  \end{enumerate}
\end{definition}

\begin{definition}[crossing neighborhood]
  We say that a neighborhood $\neighborhood^+$ {\em can perform crossing} if there exists a crossing configuration in the abelian sandpile model on $G^{\neighborhood^+}$.
\end{definition}

Figure \ref{fig:crossing-vn2} shows an example of crossing configuration for von Neumann neighborhood of radius two.

\begin{figure}
  \centering \raisebox{1.5cm}{\includegraphics{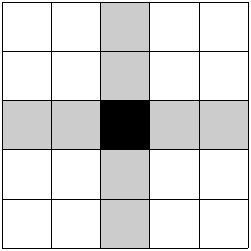}} \hspace*{2cm} \includegraphics{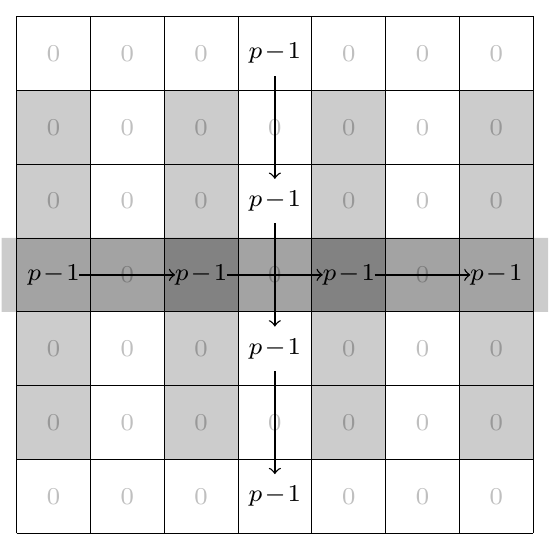}
  \caption{Von Neumann neighborhood (left) is greyed, with the cell at $(0,0)$ in black. Crossing configuration (right) for von Neumann neighborhood of radius two, of size $7 \times 7$ with vectors $\vec{n}=\vec{e}=\vec{s}=\vec{w}=(0,0,0,1,0,0,0) \in E_7$. Neighborhoods of firing cells in the west to east transport are greyed: we can notice that no vertex of the north to south transport will be fired, hence it is isolated.}
  \label{fig:crossing-vn2}
\end{figure}

\begin{definition}[shape ultimately crossing]
  We say that a shape $s^+$ {\em can ultimately perform crossing} if there exists a ratio $r_0 \in \R$ such that for all $r \in \R$, $r \geq r_0$, the neighborhood $\neighborhood^+_{s^+,r}$ can perform crossing.
\end{definition}

As mentioned at the beginning of this subsection, the definition of crossing configuration can be generalized as follows.

\begin{remark}
  Crossings can be performed in different orientations (not necessarily from the north border to the south border, and from the west border to the east border), the important property of the chosen borders is that the crossing comes from two {\em adjacent borders}, and escapes toward the two {\em mirror borders} (the mirror of north being south, the mirror of west being east, and reciprocally). It can also be delimited by a rectangle of size $n \times m$ for some integers $n$ and $m$, instead of a square.
\end{remark}

Adding one grain on a border of some stable configuration ensures that the dynamics converges in linear time in the size of the stable configuration, as stated in the next lemma.

\begin{lemma}
  \label{lemma:one}
  Let $c$ be a finite configuration of size $n \times m$, then for any $\vec{w} \in E_n$, every vertex is fired at most once during the evolution from $c+W(\vec{w})$ to a stable configuration.
\end{lemma}

Lemma \ref{lemma:converge} ensures that $c+W(\vec{w})$ converges to a stable configuration.

\begin{proof}
  The result comes from the following invariant: {\em (i)} any vertex inside the rectangle of size $n \times m$ is fired at most once; {\em (ii)} any vertex outside the rectangle of size $n \times m$ is never fired. Indeed, the invariant is initially verified, and by induction on the time $t$: {\em (i)} according to Equation (\ref{eq:sandpile}), a vertex needs to receive at least $p+1$ grains from its in-neighbors in order to fire twice. However, by induction hypothesis any vertex inside the rectangle receives at most $p$ grains, and vertices on the border receive at most $p-1$ grains, therefore even the vertex that receives the additional grain (from $W(\vec{w})$) does not receive enough grains to fire twice; {\em (ii)} vertices outside the rectangle have initially no grain and receive at most one grain, so they cannot fire.
\end{proof}

%%%%%%%%%%%%%%%%%%%%%%%%%%
\section{Known results}
\label{s:known}

As mentioned in the introduction, proofs of $\Poly$-completeness via reductions from MCVP relate the ability to perform crossing to the computational complexity of the prediction problem.

Regarding the classical neighborhoods of von Neumann (in dimension $d$ each cells has $2d$ neighbors corresponding to the two direct neighbors in each dimension, for example in dimension two the four neighbors are the north, east, south, and west cells) and Moore (von Neumann plus the diagonal cells, hence defining a square in two dimensions, a cube in three dimensions, and an hypercube in upper dimensions), it is known that the prediction problem is in $\NC$ in dimension one (\cite{1999-Moore-complexitySandpiles}), and $\Poly$-complete in dimension at least three (\cite{2006-Goles-CrossingInfo2DBTWSandpile}, via a reduction from MCVP in which it is proven that they can perform crossing). Whether their prediction problem is in $\NC$ or $\Poly$-complete in dimension two is an open question, though we know that they cannot perform crossing (\cite{2006-Goles-CrossingInfo2DBTWSandpile}).

More general neighborhoods have also been studied, such as Kadanoff sandpile models for which it has been proven that the prediction problem is in $\NC$ in dimension one (\cite{2010-GolesMartin-KSPMAP}, improved in \cite{2014-FormentiPerrotRemila-KadanoffSandpileAvalancheProblemDimensionOne} and generalized to any decreasing sandpile model in \cite{2017-FormentiPerrotRemila-DecreasingSandpileAvalancheProblemDimensionOne}), and $\Poly$-complete in dimension two when the radius is at least two (via a reduction from MCVP in which it is proven that it can perform crossing).

Threshold automata (including the majority cellular automata on von Neumann neighborhood in dimension two, which prediction problem is also not known to be in $\NC$ or $\Poly$-complete) are closely related, it has been proven that it is possible to perform crossing on undirected planar graphs of degree at most five (\cite{2013-GolesMontealegreTodinca-ComplexityBootstrapingPercolation}, hence hinting that degree four regular graph, {\em i.e.} such that $V=\Z^2$, is the most relevant case of study). The link between the ability to perform crossing and the $\Poly$-completeness of the prediction problem has been formally stated in \cite{2017-GolesMontealegrePerrotTheyssier-ComplexityDimensionTwoMajorityCellularAutomata}.

\section{Study of neighborhood, shape and crossing}
\label{s:study}

\subsection{Distinct firing graphs}

A firing graph is a useful representation of the meaningful information about a crossing configuration: which vertices fires, and which vertices trigger the firing of other vertices.

\begin{definition}[firing graph, from \cite{2006-Goles-CrossingInfo2DBTWSandpile}]
  Given a crossing configuration $c$ with vectors $\vec{n},\vec{e},\vec{s},\vec{w}$, we define the two {\em firing graphs} $G_{we}=(V_{we},A_{we})$, $G_{ns}=(V_{ns},A_{ns})$ as the directed graphs such that:
  \begin{itemize}
    \item $V_{we}$ (resp. $V_{ns}$) is the set of fired vertices in $c+W(\vec{w})$ (resp. $c+N(\vec{n})$);
    \item there is an arc $(v_1,v_2) \in A_{we}$ (resp. $\in A_{ns}$) when $v_1,v_2 \in V_{we}$ (resp. $\in V_{ns}$) and $v_1$ is fired strictly before $v_2$.
  \end{itemize}
\end{definition}

In this section we make some notations a little more precise, by subscripting the degree and set of neighbors with the digraph it is relative to. For example $\degree^+_G(v)$ denotes the out-degree of vertex $v$ in digraph $G$.

The following result is correct on all Eulerian digraph $G$ ({\em i.e.} a digraph such that $\degree^+_{G}(v) = \degree^-_{G}(v)$ for all vertex $v$), which includes the case of a uniform neighborhood on the grid $\Z^2$.

\begin{proposition}\label{prop:distinct}
  Given an Eulerian digraph $G$ for the abelian sandpile model, if there exists a crossing then there exists a crossing with firing graphs $G'_1=(V'_1,A'_1)$ and $G'_2=(V'_2,A'_2)$ such that $V'_1 \cap V'_2 = \emptyset$.
\end{proposition}

\begin{proof}
  The proof is constructive and follows a simple idea: if a vertex is part of both firing graphs, then it is not useful to perform the crossing, and we can remove it from both firing graphs.
  
  {\bf Goal.} Let $c$ be a configuration which is a crossing, and $G_1=(V_1,A_1), G_2=(V_2,A_2)$ its two firing graphs. We will explain how to construct a configuration $c'$ such that the respective firing graphs $G'_1=(V'_1,A'_1)$ and $G'_2=(V'_2,A'_2)$ verify:
  \begin{itemize}
    \item $V'_1=V_1 \setminus (V_1 \cap V_2)$;
    \item $V'_2=V_2 \setminus (V_1 \cap V_2)$.
  \end{itemize}
  This ensures that $V'_1 \cap V'_2 = \emptyset$, the expected result.

  {\bf Construction.} The construction applies two kinds of modifications to the original crossing $c$: it removes all the grains from vertices in the intersection of $G_1$ and $G_2$ so that they are not fired anymore, and adds more sand to their out-neighbors so that the remaining vertices remain fired. Formally, the configuration $c'$ is identical to the configuration $c$, except:
  \begin{itemize}
    \item for all $v \in V_1 \cap V_2$ we set $c'(v)=0$;
    \item for all $v \in \left( \bigcup\limits_{v \in V_1 \cap V_2} \neighbors^+_{G_1}(v) \right) \setminus \left( \bigcup\limits_{v \in V_1 \cap V_2} \neighbors^+_{G_2}(v) \right)$,\\[.5em]
      we set $c'(v)=c(v)+\vert \neighbors^-_{G_1}(v) \cap (V_1 \cap V_2) \vert$;
    \item for all $v \in \left( \bigcup\limits_{v \in V_1 \cap V_2} \neighbors^+_{G_2}(v) \right) \setminus \left( \bigcup\limits_{v \in V_1 \cap V_2} \neighbors^+_{G_1}(v) \right)$,\\[.5em]
      we set $c'(v)=c(v)+\vert \neighbors^-_{G_2}(v) \cap (V_1 \cap V_2) \vert$.
    %\item If a vertex $v$ belongs to more that one of these sets, it means $v \in V_1 \cap V_2$ and we set $c'(v)=0$.
  \end{itemize}

  Let us now prove that $c'$ is such that its two firing graphs $G'_1$ and $G'_2$ verify the two claims, via the combination of the following three facts.
  
  {\bf Fact 1.} It is clear that no new vertex is fired: $V'_1 \subseteq V_1$ and $V'_2 \subseteq V_2$.
 
  {\bf Fact 2.} The vertices of $V_1 \cap V_2$ are not fired in $G'_1$ nor $G'_2$:
  $$V'_1 \cap (V_1 \cap V_2) = \emptyset \text{ and } V'_2 \cap (V_1 \cap V_2) = \emptyset.$$
  Let $v \in V_1 \cap V_2$ (then $c'(v)=0$), there are two cases.\\
  {\em Case 1: $\neighbors^-_{G_1}(v) < \neighbors^-_G(v)$ and $\neighbors^-_{G_2}(v) < \neighbors^-_G(v)$.} The claim is straightforward from the fact that we set $c'(v)=0$: vertex $v$ will not receive enough grains to fire (from Fact 1).\\
  {\em Case 2: $\neighbors^-_{G_1}(v) = \neighbors^-_G(v)$ or $\neighbors^-_{G_2}(v) = \neighbors^-_G(v)$.} Without loss of generality, let us suppose that $\neighbors^-_{G_1}(v) = \neighbors^-_G(v)$. Since all in-neighbors of $v$ in $G$ are fired in $G_1$, then there exists at least one vertex $v'$ that is an in-neighbor of $v$ in both $G_1$ and $G_2$ ($v' \in \neighbors^-_{G_1}(v) \cap \neighbors^-_{G_2}(v)$), otherwise $\neighbors^-_{G_2}(v)=\emptyset$ and $v$ would not be fired in $G_2$. Hence we have $v' \in V_1 \cap V_2$, and consequently $c'(v')=0$. Now, $v$ is fired only if $v'$ is fired strictly before it. Indeed, since $c'(v)=0$ and the graph is Eulerian ($\degree^+_{G}(v) = \degree^-_{G}(v)$), $v$ needs all its in-neighbors to fire strictly before it. The same reasoning applies to $v'$: either {\em (i)} not all its in-neighbors are fired in $G_1$ and $v'$ cannot fire (the claim holds), or {\em (ii)} all its in-neighbors are fired in $G_1$ and $v'$ requires another $v''$ to be fired strictly before it. Continuing this process, either we encounter a vertex in case {\em (i)}, or, as there is a finite number of vertices, the reasoning {\em (ii)} eventually involves twice the same vertex, creating a directed cycle of vertices which must all be fired strictly before their ancestor, which is impossible. The conclusion is that none of these vertices is fired. 

  {\bf Fact 3.} The vertices of $V_1$ (resp. $V_2$) which do not belong to $V_1 \cap V_2$ are still firing in $G'_1$ (resp. $G'_2$):
  $$V_1 \setminus (V_1 \cap V_2) \subseteq V'_1 \text{ and } V_2 \setminus (V_1 \cap V_2) \subseteq V'_2.$$
  By induction on the number of time steps required to fire all vertices of the firing graph $G_1$, we can see that each vertex of $v \in V_1 \setminus (V_1 \cap V_2)$ will also be fired in $G'_1$. Indeed, we can compute that the $\vert \neighbors^-_{G_1}(v) \cap (V_1 \cap V_2) \vert$ additional grains put on $v$ in configuration $c'$ will compensate for the in-neighbors of $v$ in $G_1$ that belong to $V_1 \cap V_2$ (which are not fired anymore in $G'_1$ according to Fact 2). Let us furthermore underline that each vertex of $V'_1$ still has an in-neighbor in $V'_1$: if $v \in V'_1$ then $v \notin V_2$ so $v$ should have at least one more in-neighbor which belongs to $V_1$ than to $V_2$, and this in-neighbor still belongs to $V'_1$ (since it belongs to $V_1 \setminus (V_1 \cap V_2)$). The argument for $G_2$, $G'_2$ is similar.
   
  {\bf Conclusion.} Finally, let us argue that $c'$ is indeed a crossing configuration. It is stable by construction (we cannot add more that $p-c(v)-1$ grains to some vertex $v$, otherwise it means that $v$ belongs to $V_1 \cap V_2$ and we set $c'(v)=0$); it is isolated because $G'_1$ and $G'_2$ are subgraphs of respectively $G_1$ and $G_2$ which were isolated (Fact 1); and it is a transporter because $G'_1$ and $G'_2$ are firing graphs and vertices on the north, east, south and west borders cannot belong to $V_1 \cap V_2$, therefore (Fact 3) $G'_1$ and $G'_2$ still connect two adjacent borders to the two mirror borders.
\end{proof}

We can restate Proposition \ref{prop:distinct} as follows: if crossing is possible, then there exists a crossing with two firing graphs which have no common firing cells. It is useful to prove that some small neighborhoods (of small size $p$) cannot perform crossing, as shown below with a different proof of the impossibility of crossing with von Neumann and Moore neighborhoods of radius one, which was proved in \cite{2006-Goles-CrossingInfo2DBTWSandpile}.

\begin{corollary}[\cite{2006-Goles-CrossingInfo2DBTWSandpile}]
  Von Neumann and Moore neighborhoods of radius one cannot cross.
\end{corollary}

\begin{proof}[Alternative proof]
Assume that the neighborhoods can perform crossing. By Proposition \ref{prop:distinct}, there exists a crossing configuration so that the two firing graphs $G_1, G_2$ are distinct. Consider any arc $(v_1,v_2)$ of $G_1$ and any arc $(h_1,h_2)$ of $G_2$. The four vertices are distinct, there are two cases as follows: $(v_1,v_2)$ crosses $(h_1,h_2)$ ({\em i.e.} segments $]h_1, h_2[$ and $]v_1,v_2[$ intersect), or $(v_1,v_2)$ does not cross $(h_1,h_2)$. Because $G_1, G_2$ cross each other, it implies that there exist an arc of $G_1$ crossing an arc of $G_2$. 

  This is impossible for von Neumann neighborhood of radius one, which contradicts the assumption.

  \begin{figure}[!h]
    \centering \includegraphics{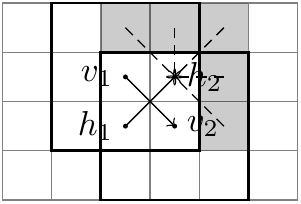}
    \caption{Only way of crossing two arcs $(h_1,h_2)$ and $(v_1,v_2)$ with Moore neighborhood of radius one. The neighborhoods of vertices $v_1$ and $v_2$ are drawn: $h_2$ belongs to both. The five darken cells are the remaining vertices that have $h_2$ in their neighborhood, among which at least two ($h_3$ and $h_4$) must be fired in $G_1$.}
    \label{fig:crossing-moore}
  \end{figure}
  
  Now consider Moore neighborhood of radius one. The arcs of $G_1$ cross that of $G_2$ with the form described in Figure \ref{fig:crossing-moore}. Suppose that $G_1$, resp $ G_2$ are started at $h_0$, resp $v_0$. Let $(v_1,v_2) \in E(G_2)$ be crossing $(h_1,h_2) \in G_1$ earliest in the crossing ({\em i.e.} $\forall \rho(h_0, h_1) \subset G_1, \forall \rho(v_0, v_1) \subset G_2$, there are no crossings between any arc $e_1 \in \rho(h_0, h_1)$ and an arc $e_2 \in \rho(v_0, v_1)$, with $\rho(u, v)$ a path from $u$ to $v$). Consider that $h_2 \in \neighborhood^{-}(v_1)$, $h_2 \in \neighborhood^{-}(v_2)$ and $G_1, G_2$ are distinct, then $G_1$ has at least three arcs to $h_2$ (including $(h_1,h_2)$), say $(h_1,h_2)$, $(h_3,h_2)$ and $(h_4,h_2)$ (obviously, $h_3, h_4 \notin V(G_2)$). With Moore neighborhood, $h_3$ and $h_4$ are on the same side with $h_2$ over line $(v_1,v_2)$. According to the definition of firing graph, $h_3$ and $h_4$ must belong to $\rho(h_0,h_1)$ because like $h_1$ they have an arc to $h_2$, but this is only possible if $\rho(h_0,h_1)$ crosses $\rho(v_0, v_1) \subset G_2$, a contradiction to the assumption that the crossing between $(v_1,v_2)$ and $(h_1,h_2)$ is the earliest.  
\end{proof}

\begin{remark}
  \label{remark:distinct}
  In the proof of Proposition \ref{prop:distinct}, given any crossing configuration $c$ of firing graphs $G_1=(V_1,A_1)$, $G_2=(V_2,A_2)$, we construct a crossing configuration $c'$ of firing graphs $G'_1=(V'_1,A'_1)$, $G'_2=(V'_2,A'_2)$ such that $G'_1$ (resp. $G'_2$) is the subgraph of $G_1$ (resp. $G_2$) induced by the set of vertices $V'_1 = V_1 \setminus (V_1 \cap V_2)$ (resp. $V'_2 = V_2 \setminus (V_1 \cap V_2)$).
\end{remark}

\subsection{Convex shapes and neighborhoods}

Proposition \ref{prop:distinct} is also convenient to give constraints on crossing configurations for some particular family of neighborhoods.

\begin{definition}[Convex shape]
  A shape $s^+$ is {\em convex} if and only if for any $u,v \in s^+$, the segment from $u$ to $v$ also belongs to $s^+$: $[u,v] \subset s^+$.
\end{definition}

\begin{definition}[Convex neighborhood]
  A neighborhood $\neighborhood^+$ is {\em convex} if and only if there exists a convex shape $s^+$ and ratio $r>0$ such that $\neighborhood^+_{s^+,r}=\neighborhood^+$.
\end{definition}

In the design crossing configurations, it is natural do try the simpler case first, which is to put $p-1$ grains on vertices we want to successively fire, and $0$ grain on other vertices. The following corollary states that this simple design of crossing configuration does not work if the neighborhood is convex.

\begin{corollary}
  For a convex neighborhood, a crossing configuration $c$ must have at least one firing vertex $v$ such that $c(v) \leq p-2$ grains.
\end{corollary}

\begin{proof}
  Let us consider a crossing configuration $c$ with two fring graphs $G_1=(V_1,A_1)$, $G_2=(V_1,A_1)$. According to Proposition \ref{prop:distinct} and Remark \ref{remark:distinct}, we know that there are two distinct fring graphs $G'_1=(V'_1,A'_1) \subseteq G_1$, $G'_2=(V'_2,A'_2) \subseteq G_2$. Then, any pair of crossing arcs between the two subgraphs is a pair of crossing arcs between $G_1, G_2$. Consider one of such pairs, say $((h_1,h_2),(v_1,v_2))$, where $(h_1,h_2 \in V'_1 \subseteq V_1$ and $v_1,v_2 \in V'_2 \subseteq V_2)$.  

  Since the neighborhood is convex, either $h_2$ is a neighbor of $v_1$, or $v_2$ is a neighbor of $h_1$. Assume that $h_2$ is a neighbor of $v_1$, as $h_2 \in V'_1 \subseteq V_1$ then $h_2 \not \in (V_1 \cap V_2)$, so $h_2 \not \in V_2$.  It means that, in configuration $c$, firing $v_1$ does not fire $h_2$, hence the number of grains at position $h_2$ is at most $p-2$.
\end{proof}

When one thinks about a shape for which crossing may be difficult to perform, a natural example would be a circular shape. Figure \ref{fig:crossing-circle} shows that given the convex shape $u$ defined as the unit disk, the neighborhood $\neighborhood^+_{u,7.25}$ can perform crossing.

\begin{figure}[!h]
  \centering \includegraphics{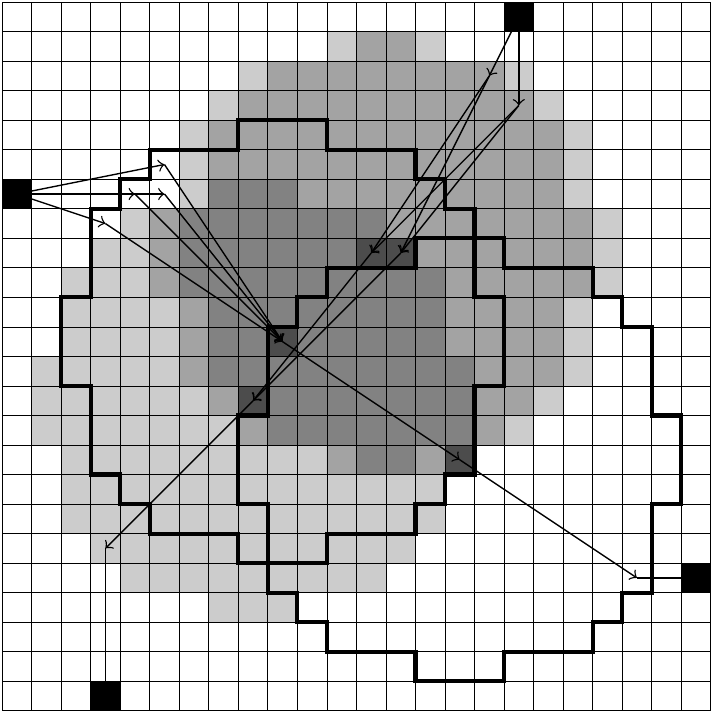}
  \caption{Crossing configuration $c$ for $\neighborhood^+_{u^+,7.25}$ with $u^+$ the unit disk shape. Five important cells of the two firing graphs are pointed with a dark grey color, and arcs of both firing graphs are drawn. Neighborhoods of the cells in the north to south firing graph are darken, and the contour of neighborhoods of the cells in the west to east firing graph are drawn. One can see that there are enough vertices inside the inverse neighborhoods (note that $\neighborhood^{-}_{u^+,7.25}=\neighborhood^+_{u^+,7.25}$) of the important cells (and outside the inverse neighborhoods of the cells that belong to the other firing graph), so that they fire only in their respective firing graph ({\em i.e.} $c(v)=\degree^-(v)$ for $v$ an important cell). Firing cells of the border are colored in black ($c(v)=p-1$ for these cells). ($c(v)=0$ elsewhere.)}
  \label{fig:crossing-circle}
\end{figure}

%An example of crossing configuration with triangular shape (see Figure\ref{fig:convex-cross}).

%\begin{figure}[!h]
%  \centering \includegraphics{convex-cross-neighborhood.pdf}\\[.5em]
%  \centering \includegraphics{convex-cross-configuration.pdf}
%  \caption{An example of convex neighborhood (top) that admits a crossing configuration (bottom). On top the black cell is fired, and the thick line delimits the set of cells that receive one grain ($p=29$ cells). The bottom picture is a crossing configuration with \textcolor{red}{unitary vectors (notion to define)}, where empty cells initially contain no grain.}
%  \label{fig:convex-cross}
%\end{figure}

%\begin{definition}[Symmetric neighborhood]
 % A neighborhood $\neighborhood^+$ is symmetric if and only if $\forall \vec{m} \in \neighborhood^+: -\vec{m} \in \neighborhood^+$.
%\end{definition}

%
\subsection{Crossing and shapes}

In this section we prove our main result: any shape can ultimately perform crossing. We first analyse how regions inside a shape scale with $r$. The following lemma is straightforward from the definition of the neighborhood of a shape (Definition \ref{def:shape}), it expresses the fact that neighboring relations are somehow preserved when we convert shapes to neighborhoods.

\begin{lemma}
  \label{lemma:partition}
  Let $s^{+1},\dots,s^{+k} \subset \R^2$ be a partition of the shape $s^+$, then $\neighborhood^+_{s^{+1},r},\dots,\neighborhood^+_{s^{+k},r}$ is a partition of the neighborhood $\neighborhood^+_{s^+,r}$.
\end{lemma}

The next lemma states that any non-flat region inside a shape can be converted (with some appropriate ratio) to an arbitrary number of discrete cells in the corresponding neighborhood.

\begin{lemma}
  \label{lemma:k}
  Let $s^+$ be a shape, and $s' \subseteq s^+$ be non-empty and non-flat. Then for any $k \in \N$, there exists a ratio $r_0 > 0$ such that for any $r \geq r_0$, $|\neighborhood^+_{s',r}| \geq k$.
\end{lemma}

\begin{proof}
  Since $s'$ is non-flat, there exists a triangle $T$ of strictly positive area inside $s'$. It follows from Definition \ref{def:shape} that the number of discrete points $\Z^2 \cap \neighborhood^+_{T,r}$ can be made arbitrarily large as $r$ increases: let $r^*>0$ be such that $T$ contains a regular triangle of size $r^*$, then $T$ contains a disk of radius $\sqrt{\frac{1}{12}}r^*$, this implies that $T$ contains a square of size $r^{**}=\sqrt{\frac{1}{6}}r^*$ (in any orientation of the square), hence $|\neighborhood^+_{T,\frac{1}{r^{**}}}| \geq 1$ and $|\neighborhood^+_{T,2^{k'}\frac{1}{r^{**}}}| \geq 4^{k'}$. The result on $\neighborhood^+_{s',r}$ follows from Lemma \ref{lemma:partition}: for any ratio $r$ we have $\neighborhood^+_{T,r} \subseteq \neighborhood^+_{s',r}$.
\end{proof}

\begin{remark}
  \label{remark:inverse}
 Lemmas \ref{lemma:partition} and \ref{lemma:k} also apply to the inverse shape $s^{-}$ and the inverse neighborhood $\neighborhood^{-}_{s^+,r}$, because the inverse neighborhood is also a neighborhood and the inverse shape is also a shape.
\end{remark}

We now prove our main result.

\begin{theorem}
  \label{theorem:shape}
  Any non-flat shape can ultimately perform crossing.
\end{theorem}

In the following construction, we choose some longest movement vectors for convenience with the arguments, but many other choices of movement vectors may allow to create crossing configurations.

\begin{proof}
  Let $s^+$ be a non-flat shape, we will show that there exists some $r_0$ such
  that for all $r \geq r_0$, $\neighborhood^+_{s^+,r}$ can perform crossing. After
  defining the setting, we will first construct the part of the finite crossing
  configuration where movement vectors (corresponding to arcs of the two firing
  graphs) do cross each other. Then we will explain how to construct the rest
  of the configuration in order to connect this crossing part to firing graphs coming from
  two adjacent borders, and to escape from the crossing part toward the two mirror
  borders.

  \paragraph{Setting.}

  This paragraph is illustrated on Figure
  \ref{fig:cross-longestvector}. Let $\vec{h}$ be a longest movement vector of
  $s^+$, $h_1=(0,0)$, and $h_2=(0,0)+\vec{h}$. The line $(h_1,h_2)$ cuts the
  shape $s^+$ into two parts, $s^1$ and $s^2$. We will choose one these two
  parts, by considering projections onto the direction orthogonal to $\vec{h}$.
  Let $\vec{v_e}$ be a vector of $s^+$ whose projection onto the direction
  orthogonal to $\vec{h}$ is the longest. Without loss of generality, let $s^2$
  be the part of $s^+$ that contains the movement vector $\vec{v_e}$. We denote
  $\vec{s^2_y}$ the projection of $\vec{v_e}$ onto the direction orthogonal to
  $\vec{h}$. 
  %We will work in the orthonormal coordinate system defined by the base
  %vectors $\vec{h}$ and $\vec{s^2_y}$, that we denote $O\vec{h}\vec{s^2_y}$.
  The fact that $\vec{h}$ and $\vec{v_e}$ have some maximality property will be
  useful in order to escape from the crossing part towards the east
  and south borders.

  \begin{figure}[!h]
    \centering \includegraphics{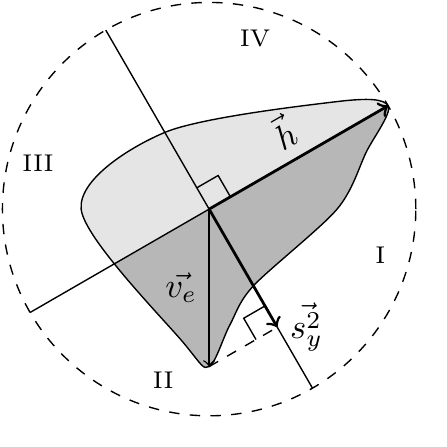}
    \caption{$\vec{h}$ is the longest movement vector of $s^+$, that cuts $s^+$
      into $s^1$ (light color, top) and $s^2$ (darker color, bottom); $\vec{v_e}$
      is a vector of $s^+$ that has the longest projection onto the direction
      orthogonal to $\vec{h}$, which we denote $\vec{s^2_y}$. Let us consider the
      orthonormal coordinate system $O\vec{h}\vec{s^2_y}$, which defines four
      quadrants pictured with roman numbers.}
    \label{fig:cross-longestvector}
  \end{figure}

  \paragraph{Crossing movement vectors in $\R^2$.}

  We now prove that there
  always exists a non-null movement vector $\vec{v} \in s^2$, not collinear with
  $\vec{h}$, that can be placed from $v_1$ to $v_2=v_1+\vec{v}$ in $\R^2$, such
  that the intersection of line segments $]v_1,v_2[$ and $]h_1,h_2[$ is not empty
  (loosely speaking, $\vec{h}$ and $\vec{v}$ do cross each other), and most
  importantly $v_1 \notin s^{-}(h_2)$, as depicted on Figure
  \ref{fig:cross-longestvector-central}). We consider two cases in order to
  find $\vec{v}$ and $v_1$ (we recall that quadrants are pictured on Figure
  \ref{fig:cross-longestvector}).

  \begin{figure}[!h] \centering
    \includegraphics{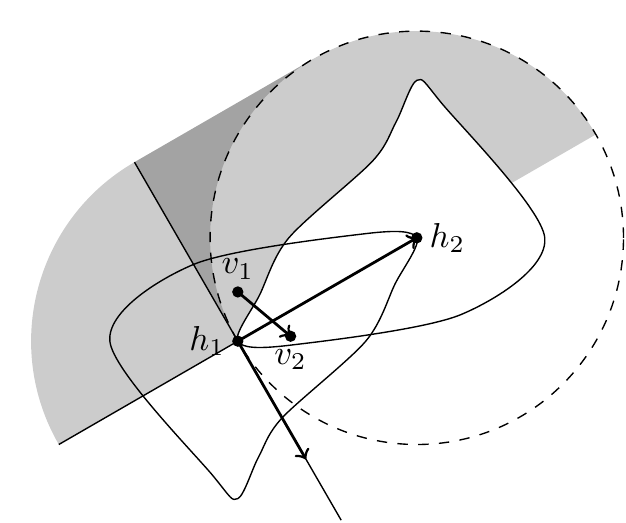}
    \caption{$\vec{h_1h_2}=\vec{h}$, the contour of $s^+(h_1)$ and
      $s^{-}(h_2)$ are drawn, and the circle of radius $|\vec{h}|$ centered at
      $h_2$ is dashed. If $s^+$ in non-flat then we can always find
      $\vec{v_1v_2}=\vec{v} \in s^2$ not collinear with $\vec{h}$ so that the
      segment $]v_1,v_2[$ crosses the segment $]h_1,h_2[$, with $v_1 \notin
      s^{-}(h_2)$. All the darken area corresponds to potential positions for $v_1$
      (outside $s^{-}(h_2)$, and so that $]v_1,v_2[$ may cross $]h_1,h_2[$ regarding
      the fact that $\vec{h}$ is a longest vector of $s^+$).}
    \label{fig:cross-longestvector-central}
  \end{figure}

  \begin{itemize}
    \item If $s^+$ has a non-flat subshape $s'$ inside the first
      quadrant, then we take $\vec{v} \in s'$ with strictly positive
      projections $\vec{v_h}$ and $\vec{v_y}$ onto the direction of $\vec{h}$
      and the direction of $\vec{s^2_y}$ (in particular $\vec{v}$ is non-null
      and not collinear with $\vec{h}$). We know that it is always possible
      to fulfill the requirements by placing $v_1$ in $\R^2$ as close as
      necessary to $h_1$, in the region of the fourth quadrant where we
      exclude  the disk of radius $|\vec{h}|$ centered at $h_2$ (see
       Figure \ref{fig:cross-longestvector-central-v}).
      We can for example place $v_1$ at position
      $(0,0)-\frac{\vec{v_y}}{2}+\epsilon\vec{h}$ for a small enough
      $\epsilon \in \R$, $\epsilon>0$.
    \item Otherwise $s^+$ is empty or flat
      inside the first quadrant, thus $\vec{v_e}$ belongs to the second
      quadrant, and $s^{-}(h_2)$ is empty inside the third quadrant (by
      symmetry of $s^{-}$ relative to $s^+$). As a consequence we can for
      example place $v_1$ at position
      $(0,0)+\frac{\vec{h}}{2}-\frac{\vec{v_e}}{2}$, so that
      $\vec{v}=\vec{v_e}$ and $v_1$ verify the requirements ($s^+$ is
      non-flat therefore $\vec{v_e}$ is non-null and not collinear with
      $\vec{h}$).
  \end{itemize}

  \begin{figure}[!h]
    \centering \includegraphics{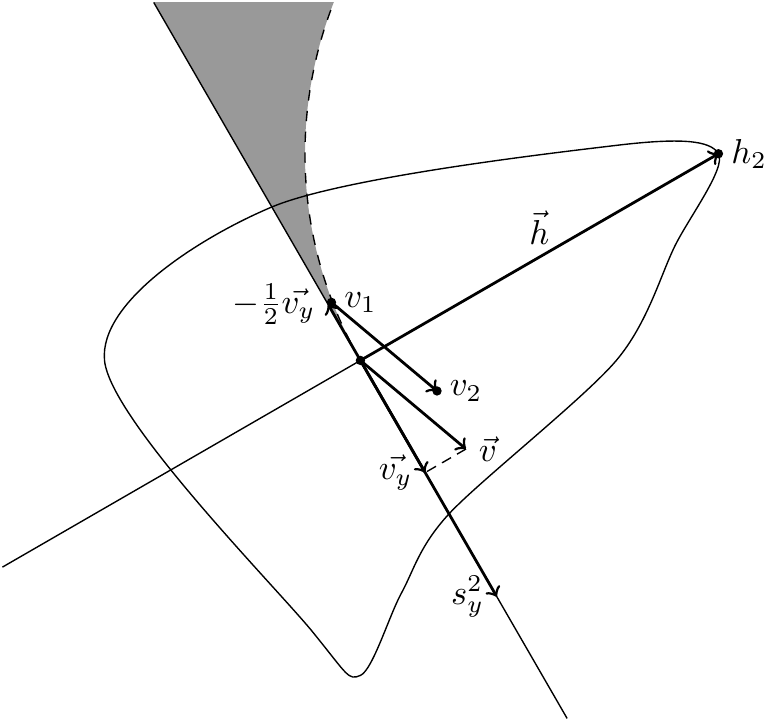}
    \caption{Choosing a vector $\vec{v}$ and a point $v_1$ at
      $(0,0)-\frac{\vec{v_y}}{2}+\epsilon\vec{h}$ such that
      $v_1 \notin s^{-}(h_2)$ (if $\epsilon>0$ is small enough then $v_1$
      belongs to the darken region), and such that segments $]h_1,h_2[$ and
      $]v_1,v_2[$ intersect. The projection of $\vec{v}$ along the axis
      $Os^2_y$ si denoted $\vec{v_y}$.}
    \label{fig:cross-longestvector-central-v}
  \end{figure}

  \paragraph{Crossing movement vectors in $\Z^2$.}

  We claim that the conditions on $\vec{v}$ allow to construct the crossing part
  of the crossing configuration as described on Figure
  \ref{fig:crossing-type3}, for $\neighborhood^+_{s^+,r}$ when $r$ is big enough.
  Indeed, as the shape is non-flat, points $h_1$ and $v_1$ can be converted to
  non-empty and non-flat subshapes $s^+_{h_1}$ and $s^+_{v_1}$ (for example by
  taking a disk of radius $\frac{\epsilon}{2}$ around each point), and we can
  apply Lemma \ref{lemma:k} to find $|H_1|=2$ and $|V_1|=4$ vertices in the
  neighborhoods corresponding to their respective subshapes,
  $\neighborhood^+_{s^+_{h_1},r}$ and $\neighborhood^+_{s^+_{v_1},r}$, when the ratio
  $r$ is bigger than some $r_1 \in \R$. Furthermore, Lemma
  \ref{lemma:partition} ensures that all the vertices in
  $\neighborhood^+_{s^+_{h_1},r}$ and $\neighborhood^+_{s^+_{v_1},r}$ preserve the neighboring
  relations of $h_1$ and $v_1$.

  Therefore we now have, for any $r \geq r_1$, a crossing part for the crossing
  configuration, as described on Figure \ref{fig:crossing-type3}. We will now
  explain how to plug to the four borders and define a proper crossing
  configuration with some vectors $\vec{n},\vec{e},\vec{s},\vec{w} \in E_n$.

  \begin{figure}[!h] \centering \includegraphics{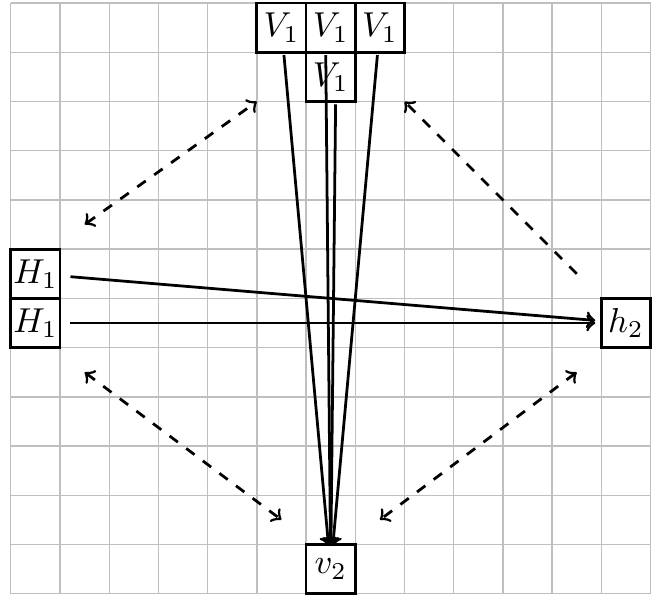}
    \caption{General form of the crossing part of our crossing configuration $c$.
      There are four regions: sets $H_1$, $V_1$,
      and individual vertices $h_2$, $v_2$. Plain arcs represent required
      neighboring relations, and dashed arcs represent other possible neighboring
      relations: we only require that none of the vertices in $V_1$ has $h_2$ in
      their neighborhood ($\neighborhood^{-}(h_2) \cap V_1 = \emptyset$). In this
      construction, we have to ensure that firing all vertices of $V_1$ and $v_2$ does not trigger the firing of $h_2$ (by acting on $H_1$), {\em i.e.}
      $|H_1| > |\{v_2\} \cap \neighborhood^{-}(h_2)|$ and for all $h_1 \in H_1$, $c(h_1) < p - | \neighborhood^{-}(h_1) \cap (V_1 \cup \{v_2\}) |$ (similarly for vertices of $V_1$).\\
      Note that all these conditions are verified if: $|H_1|=2$, $|V_1|=4$, $c(h_1)=p-6$ for all $h_1 \in H_1$, $c(h_2)=p-2$,
      $c(v_1)=p-4$ for all $v_1 \in V_1$, $c(v_2)=p-4$.}
    \label{fig:crossing-type3}
  \end{figure}

  \paragraph{Coming from two adjacent borders.}

  Let us now construct the part of the crossing configuration that connects (in
  their respective firing graphs) two adjacent borders to vertices of the
  sets $H_1$ and $V_1$. This can simply be achieved by using the movement
  vectors $\vec{h}$ and $\vec{v_e}$, respectively (see Figure \ref{fig:cross-longestvector-rest}).

  \begin{figure}[!h]
    \centering \includegraphics{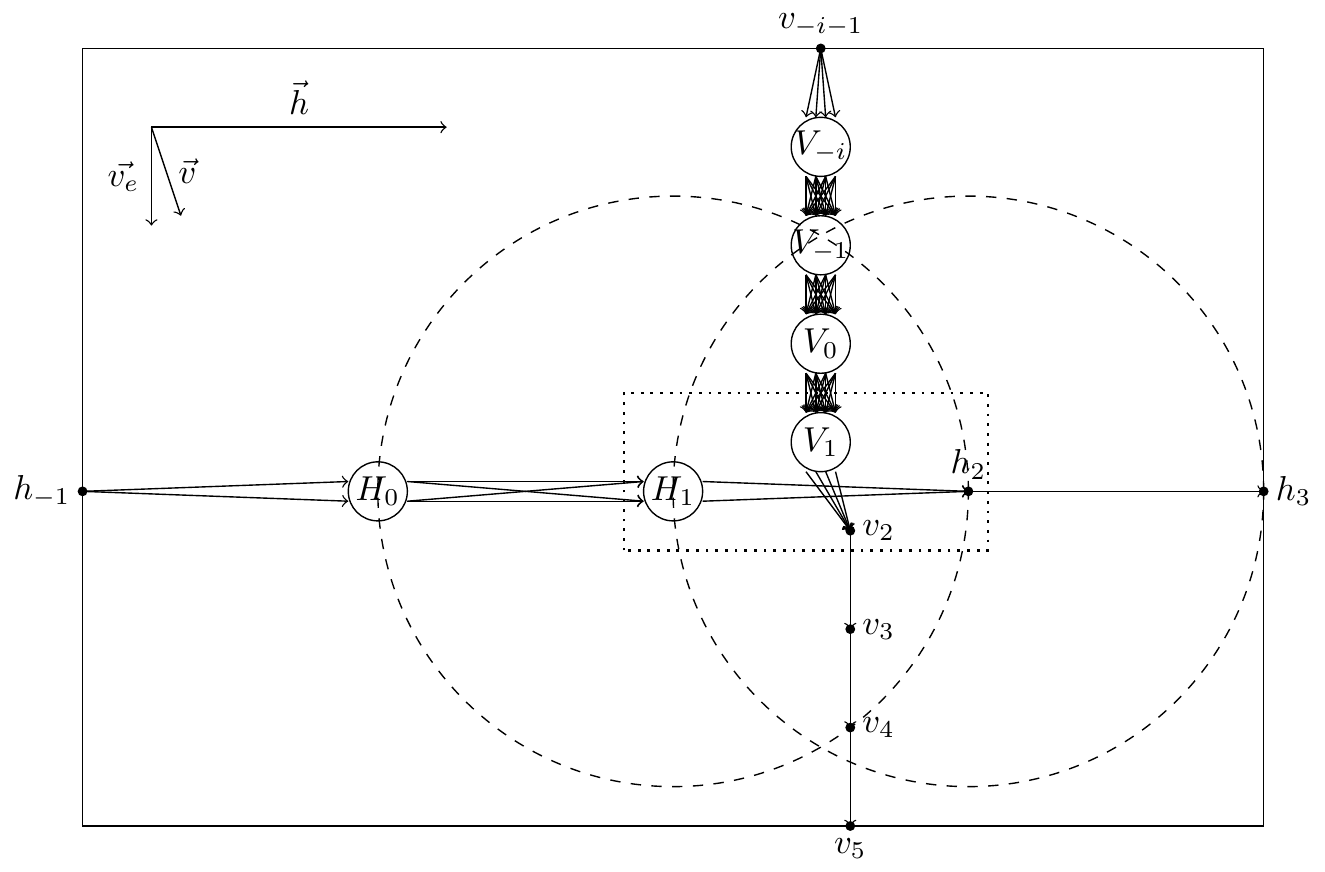}
    \caption{Global illustration of the crossing configuration. The crossing part, which uses the movement vector $\vec{v}$, is dotted. To come from two adjacent borders and escape toward the two mirror borders, the horizontal and vertical firing graphs respectively use the movement vectors $\vec{h}$ and $\vec{v_e}$.}
    \label{fig:cross-longestvector-rest}
  \end{figure}

  We construct the configuration in the reverse direction: starting from $H_1$ backward to a border, in three steps. 
  {\bf Step 1:} in $\R^2$, we consider the point $h_0$ at coordinate
  $h_1-\vec{h}$ and a non-flat subshape $s^+_0$ containing $h_0$. By Lemma \ref{lemma:k}, there exists $r_2 \in \R$ such that for any ratio $r \geq r_2$, $|\neighborhood^+_{s^+_0,r}| \geq 6$ and $\forall h_i \in
  \neighborhood^+_{s^+_0,r}$ we have $H_1 \subseteq \neighborhood^+_{s^+,r}(h_i)$ (recall that vertices of $H_1$ have $p-6$ grains). Let $H_0=\neighborhood^+_{s^+_0,r}$.
  Since $\vec{h}$ is a longest vector, $H_0$ will not
  interfere with the rest of the crossing, {\em i.e.}
  $$
    \left(\bigcup_{h_0 \in
    H_0} \neighborhood^+_{s^+,r}(h_0) \right) \cap \left( V_1 \cup \{v_2\} \cup
    \{h_2\} \right) = \emptyset.
  $$
  We place $p-1$ grains in the vertices of the
  set $H_0$. 
  {\bf Step 2:}
  we can now choose one vertex $h_{-1}$ in the direction of $-\vec{h}$ such that $H_0 \subseteq \neighborhood^+_{s^+,r}(h_{-1})$. 
  The third
  step will be explained thereafter.
  
  A similar construction can be achieved for $V_1$ using the direction given by
  $-\vec{v_e}$, using the maximality of $\vec{v_e}$ in the direction orthogonal to
  $\vec{h}$.
  % the new vertices won't interfere with the rest of the crossing, as
  %it was the case for $H_0$.
  The difference with the previous case is that we may need to apply few times the first step, giving a sequence of points
  $v_0,v_{-1},v_{-2},\dots$ corresponding to sets $V_0,V_{-1},V_{-2},\dots$ of
  vertices on which we put $p-4$ grains, until we have some point $v_{-i}$
  outside the union of the two disks of radius $|\vec{h}|$ centered at $h_1$
  and $h_2$. The next point, $v_{-i-1} \in \R^2$ can safely constitute the second
  step, {\em i.e.} we can take only one vertex $v_{-i-1} \in \Z^2$ such that
  $V_{-i} \subseteq \neighborhood^+_{s^+,r}(v_{-i-1})$. Let $r_3$ be the maximum of ratios given by
  applications of Lemma \ref{lemma:k} in this case.

  {\bf Step 3:} we now have two vertices $h_{-1}$ and $v_{-i}$, that we can
  consider as part of two adjacent borders given by the directions of
  $-\vec{h}$ and $-\vec{v_2}$, respectively (we may again use the fact that the
  shape is non-flat in order to avoid any problem, for example if $-\vec{h}$
  points in a direction collinear with $\vec{x}+\vec{y}$, {\em i.e.} towards an
  angle between two borders rather than one border). This defines two vectors
  of $E_n$ corresponding to two adjacent borders.

  \paragraph{Escaping toward the two mirror borders.}

  Escaping from the crossing part towards the two mirror borders is very similar
  to coming to from the previous two adjacent borders: we use the movement
  vectors $\vec{h}$ and $\vec{v_e}$ that still do not interfere with the rest
  of the crossing configuration thanks to their maximality property, and define
  as many vertices as necessary on which we place $p-1$, until we reach the two
  mirror borders given by the directions of $\vec{h}$ and $\vec{v_2}$, thus
  defining two vectors of $E_n$ corresponding to the two mirror borders (see again Figure \ref{fig:cross-longestvector-rest}). Let
  $r_4 \in \R$ be the maximum of ratios given by applications of Lemma
  \ref{lemma:k} in this case.

  \paragraph{Conclusion.}

  We have first constructed a crossing part where arcs of the respective firing
  graphs do cross, and in a second part we constructed the rest of the
  configuration in order to connect the firing graphs from two adjacent borders
  to the two incoming endpoints of the crossing part, and finally we constructed
  the rest of the configuration in order to connect the firing graphs from the
  two outgoing endpoints of the crossing part to the two mirror borders of the
  crossing configurations. This configuration is finite, stable, and transports
  from two adjacent borders to the two mirror borders, with isolation, {\em i.e.} it
  is a crossing configuration.
  
  Let $r_0=\max\{r_1,r_2,r_3,r_4\}$, we have therefore achieved to prove
  that for any ratio $r \geq r_0$, the neighborhood $\neighborhood^+_{s^+,r}$ can
  perform crossing.
\end{proof}

 %\textcolor{red}{Question: we found that shape $s$ is circle, then $s$ can make crossing when a neighborhood with radius $\geq 8$; my prediction is that any neighborhood whose longest movement is greater than 8, can make crossing??}

%%%%%%%%%%%%%%%%%%%%%%%%%%
\section{Conclusion and perspective}

After giving a precise definition of crossing configurations in the abelian sandpile model on $\Z^2$ with uniform neighborhood, we have proven that the corresponding firing graphs can always be chosen to be distinct. We have seen that this fact has consequences on the impossibility to perform crossing for some neighborhoods with short movement vectors, and that crossing configurations with convex neighborhoods require some involved constructions with firing cells having at least two in-neighbors in the firing graphs. We have presented an example of crossing configuration with a circular shape, and finally proved the main result that any shape can ultimately perform crossing (Theorem \ref{theorem:shape}).

As a consequence of Theorem \ref{theorem:shape}, the conditions on a neighborhood such that it cannot perform crossing cannot be expressed in continuous terms, but are intrinsically linked to the discreteness of neighborhoods. It remains to find such conditions, {\em i.e.} to characterize the class of neighborhoods that cannot perform crossing. More generally, what can be said on the set of neighborhoods that cannot perform crossing? It would also be interesting to have an algorithm to decide whether a given neighborhood can perform crossing or not, since the decidability of this question has not yet been established.

%%%%%%%%%%%%%%%%%%%%%%%%%%
\section{Acknowledgment}

This work received support from FRIIAM research federation (CNRS FR 3513), and JCJC INS2I 2017 project CGETA.

\bibliographystyle{plain}
\bibliography{biblio}

\end{document}